\documentclass[11pt,a4paper,leqno]{amsart}

\usepackage{amssymb,amsmath,latexsym,amsthm,enumerate}

\theoremstyle{thmit} % Numbered and Italic
\newtheorem{thm}{Theorem}[section]
\newtheorem{lem}[thm]{Lemma}
\newtheorem{cor}[thm]{Corollary}
\newtheorem{prop}[thm]{Proposition}

\usepackage{mathrsfs}

\theoremstyle{definition}
\newtheorem{remark}[thm]{Remark}

\newtheorem{assumption}[thm]{Assumption}

\makeatletter
\@namedef{subjclassname@2020}{%
  \textup{2020} Mathematics Subject Classification}
\makeatother

\numberwithin{equation}{section}

\DeclareMathOperator*{\Res}{Res}

\allowdisplaybreaks

\begin{document} 

\title{Double Dirichlet series associated with arithmetic functions II}
\author{Kohji Matsumoto}
\address{K. Matsumoto: Graduate School of Mathematics, 
Nagoya University, 
Chikusa-ku, Nagoya 464-8602, Japan}
\email{kohjimat@math.nagoya-u.ac.jp}
\author{Hirofumi Tsumura}
\address{H. Tsumura: Department of Mathematical Sciences, 
Tokyo Metropolitan University, 
1-1, Minami-Ohsawa, Hachiouji, Tokyo 
192-0397, Japan}
\email{tsumura@tmu.ac.jp}

\subjclass[2020]{Primary 11M41; Secondary 11M06, 11M26}

\keywords{Multiple Dirichlet series; Riemann zeta function; von Mangoldt function; M{\"o}bius function}

\thanks{This work was supported by Japan Society for the Promotion of Science, Grant-in-Aid for Scientific Research No. 25287002 (K. Matsumoto) and No. 22K03267 (H. Tsumura).}

\maketitle

\begin{abstract}
This paper is a continuation of our previous work on double Dirichlet series associated with arithmetic functions such as the von Mangoldt function, the M\"obius function, and so on. We consider the analytic behaviour around the non-positive integer points on singularity sets which are points of indeterminacy. In particular, we show a certain reciprocity law of their residues.
Also on this occasion we correct some inaccuracies in our previous paper.
\end{abstract}

\bigskip
\section{Introduction}\label{sec-1}
Let $\mathbb{N}$ be the set of positive integers, $\mathbb{N}_0=\mathbb{N} \cup \{ 0\}$, 
$\mathbb{Z}$ the set of integers,
$\mathbb{P}$ the set of prime numbers,
$\mathbb{R}$ the field of real numbers, $\mathbb{C}$ the field of complex numbers and $i:=\sqrt {-1}$.

In our previous paper \cite{MNT2021}, we considered general double Dirichlet series 
of the form
\begin{align}\label{def:mDS}
\Phi _2 &(s_1,s_2;\alpha _1,\alpha _2)=\sum _{m_1=1}^\infty \sum _{m_2=1}^\infty \frac {\alpha _1 (m_1)\alpha _2 (m_2)}{m_1^{s_1}(m_1+m_2)^{s_2}},%\nonumber
\end{align}
where $\alpha_1,\,\alpha_2$ are arithmetic functions satisfying certain conditions.

If $\alpha_k$ ($k=1,2$) are periodic, then the study of
$\Phi _2 (s_1,s_2;\alpha _1,\alpha _2)$ essentially reduces to
the case of double Hurwitz zeta function defined by 
\begin{align*}
& \zeta _r (s_1,s_2;d_1,d_2)=\sum _{m_1=0}^\infty \sum _{m_2=0}^\infty \frac{1}{(m_1+d_1)^{s_1}(m_1+m_2+d_1+d_2)^{s_2}}  %\label{EZ}
\end{align*}
for positive rational numbers $d_1$ and $d_2$, which has been investigated classically. 

A more general situation was studied by Tanigawa and the first-named author
\cite{MatsumotoTanigawa}.   They considered a certain class of arithmetic functions $\{ \alpha\,:\mathbb{N} \to \mathbb{C}\}$, the Dirichlet series associated with $\alpha$:
\begin{equation}\label{def:DS}
\Phi (s;\alpha )=\sum _{n=1}^\infty \frac {\alpha (n)}{n^s}
\end{equation}
and its multiple version as follows. 
Let $\mathcal{A}$ be the set of arithmetic functions satisfying the following three 
conditions: If $\alpha\in\mathcal{A}$, then
\begin{itemize}
\item[(I)] $\Phi (s;\alpha)$ is absolutely convergent for $\Re s>\delta=\delta(\alpha)(>0)$;
\item[(II)] $\Phi (s;\alpha)$ can be continued meromorphically to the whole plane $\mathbb{C}$, holomorphic except for a possible pole (of order at most 1) at $s=\delta$; 
\item[(III)] in any fixed strip $\sigma _1\leq \sigma \leq \sigma _2$, $\Phi (\sigma +it;\alpha)=O(|t|^A)$ holds as $|t|\to \infty$, where $A$ is a non-negative constant.
\end{itemize}
Then they showed that 
for arithmetic functions $\alpha_1,\ldots,\alpha_r$ belonging to $\mathcal{A}$, 
\begin{align*}
&\Phi _r (s_1,\dots ,s_r;\alpha _1,\dots ,\alpha _r)\\
&=\sum _{m_1=1}^\infty \cdots\sum _{m_r=1}^\infty \frac {\alpha _1 (m_1)
\alpha_2(m_2)\cdots
\alpha _r (m_r)}{m_1^{s_1}(m_1+m_2)^{s_2}\cdots(m_1+\cdots+m_r)^{s_r}}
\end{align*}
can be continued meromorphically to the whole space $\mathbb{C}^r$, and location of its possible singularities can be described explicitly. In particular, if all $\Phi (s;\alpha_k)$'s are entire, then $\Phi _r (s_1,\dots ,s_r;\alpha _1,\dots ,\alpha _r)$ is also entire 
(see \cite[Theorem 1]{MatsumotoTanigawa}). 
Further it is easy to extend the results in \cite{MatsumotoTanigawa} to the case when
the condition (II) is replaced by
(II)': $\Phi(s;\alpha)$ is continued to $\mathbb{C}$ and holomorphic except for finitely many poles.

As another example, Egami and the first-named author \cite{MatsumotoEgami} considered the double series associated with the von Mangoldt function. Let $\zeta(s)$ be the Riemann zeta-function and denote by 
$\{\rho_{n}\}_{n\geq 1}$ the non-trivial zeros of $\zeta(s)$ numbered by the size of absolute values of their imaginary parts. (As for the pair of complex conjugates, we first count the zero whose imaginary part is positive, and then count negative.) 
Let 
$\Lambda$ be the von Mangoldt function 
defined by
\begin{eqnarray*}
\Lambda (n)=\left\{ \begin{array}{ll}
\log p & (n=p^m\ \textrm{for}\ p\in \mathbb{P},\ m\in \mathbb{N}) \\
0 & ({\rm otherwise}) \\
\end{array} \right. .
\end{eqnarray*}
Then it is well-known that
\begin{equation*}
\Phi (s;\Lambda )=\sum _{n=1}^\infty \frac {\Lambda (n)}{n^s}=-\frac {\zeta '(s)}{\zeta (s)}
\end{equation*} 
(see \cite[$\S$ 1.1]{Titch}). We denote this function by $M(s)$. 
We see that $M(s)$ has poles at $s=1$, $s=-2m\ (m\in \mathbb{N})$ and $s=\rho_l$ $(l\in \mathbb{N})$, hence does not satisfy the assumption (II), or even (II)'.    That is, 
$\Lambda\notin\mathcal{A}$.
As a double version of $M(s)$, Egami and the first-named author \cite{MatsumotoEgami} 
introduced 
\begin{eqnarray*}
{\mathcal M}_2(s)&=&\sum _{m_1=1}^\infty \sum _{m_2=1}^\infty \frac {\Lambda (m_1)\Lambda (m_2)}{(m_1+m_2)^s}\ (=\Phi _2(0,s;\Lambda ,\Lambda ))
\end{eqnarray*}
which can be written as $\sum _{n=1}^\infty G_2(n)n^{-s}$,
where
\begin{equation*}
G_2(n)=\sum _{m_1+m_2=n}\Lambda (m_1)\Lambda (m_2).
\end{equation*}
We emphasize that ${\mathcal M}_2(s)$ is closely connected to the Goldbach conjecture, that is, $G_2(n)>0$ for all even $n\geq 4$ (cf. \cite{BHMS}). 

Since $\Lambda\notin\mathcal{A}$, we cannot apply the result in \cite{MatsumotoTanigawa}. In fact, the line $\Re s=1$ can be shown to be the natural boundary of ${\mathcal M}_2(s)$ under some plausible assumptions, hence ${\mathcal M}_2(s)$ would not be continued meromorphically to the whole complex plane $\mathbb{C}$ (see \cite{MatsumotoEgami} and \cite{BSP11}). 

In \cite{MNT2021}, we considered a certain family of the double series defined by \eqref{def:mDS} with $\alpha_1\equiv 1$, $\alpha_2\notin\mathcal{A}$, but
can be continued meromorphically to the whole space.

The first example is 
\begin{equation}\label{def:phi(1,l)}
\Phi _2 (s_1,s_2;1,\Lambda )=\sum _{m_1=1}^\infty \sum _{m_2=1}^\infty \frac {\Lambda (m_2)}{m_1^{s_1} (m_1+m_2)^{s_2}}.
\end{equation}
The right-hand side of (\ref{def:phi(1,l)}) is absolutely convergent for $\Re s_2>1,\Re (s_1+s_2)>2$. Furthermore we can prove that $\Phi _2(s_1,s_2;1,\Lambda )$ can be continued meromorphically to the whole space $\mathbb{C}^2$ (see Theorem \ref{state1}).   We also determine
the location of singularities of $\Phi _2(s_1,s_2;1,\Lambda )$ (see 
Corollary \ref{C-2-1}).

The second example is, for any $\Phi (s;\alpha)$ with $\alpha\in\mathcal{A}$, 
\begin{equation}\label{def_double_tilde}
\Phi _2 (s_1,s_2;1,\widetilde{\alpha} )=\sum _{m_1=1}^\infty \sum _{m_2=1}^\infty \frac {\widetilde{\alpha} (m_2)}{m_1^{s_1} (m_1+m_2)^{s_2}},
\end{equation}
where $\widetilde{\alpha}:\,\mathbb{N} \to \mathbb{C}$ is defined by
\begin{equation}\label{def_tilde_alpha}
\frac{\Phi (s;\alpha)}{\zeta(s)}=\sum _{n=1}^\infty \frac {\widetilde{\alpha}(n)}{n^s}\left(=:\Phi(s;\widetilde{\alpha})\right).
\end{equation}
In \cite{MNT2021} it was shown that $\Phi _2 (s_1,s_2;1,\widetilde{\alpha} )$ is absolutely convergent in the region
\begin{equation}
\left\{(s_1,s_2)\in \mathbb{C}^2 \mid \Re s_2>\max\{1,\delta\},\ \Re (s_1+s_2)>\max\{ 2,1+\delta\}\right\}\label{conv-region-intro}
\end{equation}
(see Lemma \ref{Prop-region}), 
and $\Phi _2(s_1,s_2;1,\widetilde{\alpha})$ can be continued meromorphically to the whole space $\mathbb{C}^2$ (see Theorem \ref{Kocchiga-Th-4-2}), under the assumption of a certain quantitative version of the simplicity conjecture of $\zeta(s)$ (see Assumption \ref{Ass-2}). 
Note that we can apply this result to the cases when $\widetilde{\alpha}$ is the M\"obius function $\mu$, the Euler totient function $\phi$, and so on. 

The points  $(u_1,u_2)$ on singularity sets are usually points of indeterminacy, 
whose ``values''
can be understood only as limit values which depend on how to choose a limiting process.
Two typical limit processes are ``regular values'' and ``reverse values'', defined by
\begin{align}
& \lim_{s_1\to u_1}\lim_{s_2\to u_2}, \label{reg-vals} \\
& \lim_{s_2\to u_2}\lim_{s_1\to u_1}, \label{rev-vals}
\end{align}
respectively.   These values were first introduced and studied by 
Akiyama, Egami and Tanigawa \cite{AET}, and Akiyama and Tanigawa \cite{AT} in the case of multiple zeta-functions of Euler-Zagier type. It should be noted that Onozuka \cite{Onozuka2013} showed that limit values of Euler-Zagier multiple zeta-functions at non-positive integer points are always convergent, under any choice of the limiting process. 

In contrast, the situation is different when we consider the reverse values of
\eqref{def:phi(1,l)} and \eqref{def_double_tilde}.

In \cite{MNT2021}, we explicitly computed the reverse values of \eqref{def:phi(1,l)} (see Proposition \ref{spv1}) and \eqref{def_double_tilde} 
with $\widetilde{\alpha}=\mu$,
when $(s_1,s_2)=(-m,-n)$ where $m,n\in \mathbb{N}_{ 0}$ with $m \equiv n \mod 2$,
and $m\geq 1$ when $n\geq 2$.
In these cases, the reverse values are still finite definite values.

In the present paper, we observe the behavior of \eqref{def:phi(1,l)} and \eqref{def_double_tilde} with respect to the limiting process for the case \eqref{rev-vals} at $(-m,-n)$ with $m\in \mathbb{Z}$ and $n\in \mathbb{N}_{0}$ such that $m \not\equiv n \mod 2$. In particular, even if $(m,n)\in \mathbb{N}_{ 0}^2$, we find that the reverse values are sometimes not convergent in this case (see Theorems \ref{Theorem-3-1} and \ref{Theorem-5-1}). First we recall their analytic behavior (see Section \ref{sec-2}). From these results, we give asymptotic formulas for reverse values of \eqref{def:mDS} (see Section \ref{sec-3}). Moreover 
we prove a certain reciprocity law of their residues which comes from the known reciprocity law for Bernoulli numbers (see Section \ref{sec-4}). Similarly, we observe the analytic behavior of \eqref{def_double_tilde} (see Section \ref{sec-5}). % Also we prove a certain reciprocity law of their residues (see \eqref{}). 
Finally we give corrigendum and addendum to the previous paper \cite{MNT2021} (see Section \ref{sec-6}).

\ 

%%%%%%%%%%%%%%%%%%%%%%%%%%%%%%%%%%%%%%%%%%%%%%%%%%%%%%%%%
\section{Known results}\label{sec-2}
%%%%%%%%%%%%%%%%%%%%%%%%%%%%%%%%%%%%%%%%%%%%%%%%%%%%%%%%%

In this section, we recall several known results for $\Phi _2 (s_1,s_2;1,\Lambda )$ and 
$\Phi _2(s_1,s_2;1,\widetilde{\alpha})$ given in \cite{MNT2021}.

In order to prove those results, we use the 
Mellin-Barnes integral formula (see, for example, \cite[Section 14.51, p.289, Corollary]{WW}):
\begin{equation}\label{MB}
(1+\lambda )^{-s}=\frac 1{2\pi i}\int _{(c)}\frac {\Gamma (s+z)\Gamma (-z)}{\Gamma (s)}\lambda ^{z}dz,
\end{equation}
where $s,\lambda \in \mathbb{C}$ with $\lambda \neq 0,|{{\rm arg}\lambda}|<\pi,\ \Re s>0$, 
$c\in \mathbb{R}$ with $-\Re s<c<0$, and the path of integration is the vertical line form $c-i\infty$ to $c+i \infty$. 

We prepare some notation as follows. Let $a_k$ and $b_l$ be 
%\begin{align*}
%& a_k=\frac 1{2\pi i}\int _{|\xi +k|=\frac 1 2} \frac {M(\xi )}{\xi +k} d\xi\quad (k\in \mathbb{N};\ k:\textrm{even}), \\ 
%& b_l=\frac 1{2\pi i}\int _{|\xi +l|=\frac 1 2} \frac {\Gamma (\xi )}{\xi +l} d\xi\qquad (l\in \mathbb{N}_0), 
%\end{align*}
constant terms of Laurent series of $M(s)=-\zeta'(s)/\zeta(s)$ at $s=-k$
($k\in\mathbb{N}$, $k:\textrm{even}$), and of $\Gamma(s)$ at $s=-l$
($l\in\mathbb{N}_0$), 
respectively.   Then we have 
\begin{align}
&a_k=\lim_{s\to -k}\ \frac{d}{ds}(s+k)M(s)=-\frac{\zeta''(-k)}{2\zeta'(-k)},\label{def-ak}\\
&b_l=\lim_{s\to -l}\ \frac{d}{ds}(s+l)\Gamma(s)=\frac{(-1)^l}{l!}\left(\sum_{j=1}^{l}\frac{1}{j}-\gamma\right), \label{def-bk}
\end{align}
where $\gamma$ is the Euler constant. 
Expressing $\Phi _2 (s_1,s_2;1,\Lambda )$ by using \eqref{MB} and shifting
the path of integration, we obtained 

\begin{thm}[\cite{MNT2021}\ Theorem 2.1] \label{state1} $\Phi _2(s_1,s_2;1,\Lambda )$ can be continued meromorphically to the
whole space $\mathbb{C}^2$ by the following expression:
\begin{align}\label{mainthm}
&\Phi _2 (s_1,s_2;1,\Lambda ) \\
&=\frac {\zeta (s_1 +s_2 -1)}{s_2 -1}-(\log 2\pi ) \zeta (s_1+s_2)\nonumber\\
&+\sum ^{N-1}_{k=1 \atop k:{\rm odd}} \binom {-s_2}k M(-k)\zeta (s_1+s_2+k)\nonumber \\
&-\sum ^{N-1}_{k=1 \atop k:{\rm even}} \bigg[ \binom {-s_2}k\left\{ \left( -{a_k}+k!b_k\right)\zeta (s_1+s_2+k)-\zeta'(s_1+s_2+k)\right\} \nonumber \\
& \qquad\qquad -\frac {1}{k!}\frac{\Gamma'(s_2+k)}{\Gamma(s_2)}\zeta(s_1+s_2+k)\bigg]\nonumber\\
&-\frac{1}{\Gamma (s_2)}\sum _{n=1}^\infty {\rm ord}(\rho _n)\Gamma (s_2-\rho _n)\Gamma (\rho _n)\zeta (s_1+s_2-\rho _n)\nonumber  \\
&+\frac{1}{2\pi i\Gamma (s_2)}\int_{(N-\eta )}\Gamma (s_2+z)\Gamma (-z)M(-z)\zeta (s_1+s_2+z)dz,\nonumber 
\end{align}
where $N\in \mathbb N$, %satisfies $N>2, N>1-\Re s_2, N>2-\Re (s_1+s_2)$, 
$\eta$ is a small positive number, and ${\rm ord}(\rho _n)$ is the order of $\rho _n$ as a zero of $\zeta(s)$.
\end{thm}

\begin{cor}[\cite{MNT2021}\ Theorem 2.2] \label{C-2-1}\ 
The list of singularities of $\Phi _2 (s_1,s_2;1,\Lambda )$ are given as follows:
%only on the subsets of $\mathbb{C}^2$ defined by one of the following equations:
\begin{align}
& s_2=1,  \label{sing-1}\\
& s_2=-l\quad (l\in \mathbb{N},\ l\geq 2),\label{sing-2}\\
& s_1+s_2=2-l\quad (l\in \mathbb{N}_0),\label{sing-3}\\
& s_2=-l+\rho_n\quad (n\in \mathbb{N},\ l\in \mathbb{N}_0),\label{sing-4}\\
&s_1+ s_2=1+\rho_n\quad (n\in \mathbb{N}).\label{sing-5}
\end{align}
\end{cor}

Next we consider a general class of double series $\Phi _2 (s_1,s_2;1,\widetilde{\alpha} )$ defined by 
\eqref{def_double_tilde}.
Since
\begin{equation}
\frac{1}{\zeta(s)}=\sum _{n=1}^\infty \frac {\mu(n)}{n^s}\quad (\Re s>1),\label{mu-Dir}
\end{equation}
where $\mu$ is 
the M{\"o}bius function 
(see \cite[$\S$ 1.1]{Titch}), %Hence, for $\Re s>\max\{1,\delta\}$, 
%we have
%$$\frac{\Phi (s;\alpha)}{\zeta(s)}=\sum _{m=1}^\infty \sum _{n=1}^\infty\frac {{\alpha}(m)\mu(n)}{(mn)^s}.$$
from \eqref{def_tilde_alpha} we have
\begin{equation}
\widetilde{\alpha}(n)=\sum_{1\leq d \leq n \atop d\mid n}\alpha\left(\frac{n}{d}\right)\mu(d). \label{Rev-form}
\end{equation} 

\begin{lem}[\cite{MNT2021}\ Proposition 3.1] \label{Prop-region}
If $\alpha$ satisfies the condition (I) of the class $\mathcal{A}$, we see that 
$\Phi _2 (s_1,s_2;1,\widetilde{\alpha} )$ is absolutely convergent in the region
\begin{equation}
\left\{(s_1,s_2)\in \mathbb{C}^2 \mid \Re s_2>\max\{1,\delta\},\ \Re (s_1+s_2)>\max\{ 2,1+\delta\}\right\}.\label{conv-region}
\end{equation}
\end{lem}

%For $\alpha\in\mathcal{A}$, 
%we assume the following condition for $\zeta(s)$.

Here we introduce the following assumption, which is plausible in view of
the Gonek-Hejhal conjecture (Gonek \cite{Gonek}, Hejhal \cite{Hejhal}).

\begin{assumption}\label{Ass-2}\ All non-trivial zeros of $\zeta(s)$ are simple 
and 
\begin{equation}\label{order:1/zeta'(rho)}
\frac 1{\zeta '(\rho _n)}=O(|{\rho _n}|^B)\quad (n\to \infty)
\end{equation}
with some constant $B> 0$. 
\end{assumption}

Let $c_k(\alpha)$ be 
%\begin{align}
%& c_h(\alpha)=\frac 1{2\pi i}\int _{|\xi +h|=\frac 1 2} \frac {\Phi(s;\alpha)}{\zeta(\xi)}\frac{1}{\xi +h} d\xi\quad (h\in \mathbb{N};\ k:\textrm{even}), \label{def-ck}
%\end{align}
the constant term of the Laurent series of 
$\Phi(s;\alpha)/\zeta (s)$ at $s=-k$ ($k\in\mathbb{N}$, $k:\mathrm{even}$). 
%We can also express that 
%\begin{align}
%&c_h=\lim_{s\to -h}\ \frac{d}{ds}\frac{s+h}{\zeta(s)}=-\frac{\zeta''(-h)}{2(\zeta'(-h))^2}\quad (h:\textrm{even}). \label{def-ck-2}
%\end{align}
Define the Bernoulli numbers $\{B_n\}_{n\geq 0}$ by
\begin{equation*}
\frac {te^t}{e^t-1}=\sum _{n=0}^\infty B_n\frac {t^n}{n!}.
\end{equation*}
Then we obtained 

\begin{thm}[\cite{MNT2021}\ (4.4)] \label{Kocchiga-Th-4-2}
Let $\alpha\in\mathcal{A}$.
Under Assumption \ref{Ass-2}, $\Phi _2(s_1,s_2;1,\widetilde{\alpha} )$ can be continued meromorphically to the whole space $\mathbb{C}^2$ by {\rm (\ref{mainthm2})} by the following expression:
\begin{align}\label{mainthm2}
&\Phi _2(s_1,s_2;1,\widetilde{\alpha} )\\
&=\frac{\Gamma(s_2-\delta)\Gamma(\delta)}{\Gamma(s_2)} \Res_{s=\delta}\ \left(\frac{\Phi(s;\alpha)}{\zeta(s)}\right)\zeta(s_1+s_2-\delta) \notag\\
& \ -2\Phi(0;\alpha)\zeta (s_1+s_2)-\sum _{k=1 \atop k:{\rm odd}}^{N-1}\binom {-s_2}k \frac {(k+1)\Phi(-k;\alpha)}{B_{k+1}}\zeta (s_1+s_2+k)\nonumber \\
&\ +\sum _{k=1\atop k:{\rm even}}^{N-1}  \bigg[ \binom {-s_2}k\bigg\{ \frac {(-1)^{k/2}2(2\pi )^{k}\Phi(-k;\alpha)}{\zeta (1+k)}\notag \\
& \ \ \times \left(b_k \zeta(s_1+s_2+k)-\frac{1}{k!}\zeta'(s_1+s_2+k)\right)+c_k(\alpha)\zeta(s_1+s_2+k)\bigg\}\nonumber \\
& \qquad -\frac{\Gamma'(s_2+k)}{\Gamma(s_2)}\frac {(-1)^{k/2}2(2\pi )^{k}\Phi(-k;\alpha)}{(k!)^2\zeta (1+k)}\zeta(s_1+s_2+k)\bigg]\notag\\
&\ +\frac 1 {\Gamma (s_2)}\sum _{n=1}^\infty \Gamma (s_2-\rho _n)\Gamma (\rho _n)\frac {\Phi(\rho_n;\alpha)}{\zeta'(\rho _n)}\zeta (s_1+s_2-\rho _n)\nonumber \\
&\ +\frac 1 {2\pi i\Gamma (s_2)}\int _{(N-\eta)}\Gamma (s_2+z)\Gamma (-z)\Phi(-z;\widetilde{\alpha})\zeta (s_1+s_2+z)dz\nonumber
\end{align}
for $N\in \mathbb N$, where $\eta$ is a small positive number, and $b_k$ is defined by \eqref{def-bk}.
\end{thm}

\begin{cor}[\cite{MNT2021}\ Theorem 4.3] \label{Th-4-2}
The list of singularities of $\Phi _2 (s_1,s_2;1,\widetilde{\alpha} )$ is given by:
\begin{equation}
\begin{split}
& s_1+s_2=1-k\quad (k\in \mathbb{N}_0),\\
& s_2=-k\quad (k\in \mathbb{N},\ k\geq 2),\\
& s_2=-l+\rho_n\quad (l\in \mathbb{N}_0,\ n\in \mathbb{N}),\\
& s_1+ s_2=1+\rho_n\quad (n\in \mathbb{N}),\\
& s_2=-l+\delta\quad  (l\in \mathbb{N}_0),\\
& s_1+ s_2=1+\delta,
\end{split}
\label{coro-2}
\end{equation}
where the last two equations are omitted when $\Phi(s;\alpha)$ has no pole or $\delta=1$. 
\end{cor}

We define the reverse value of $\Phi _2 (s_1,s_2;\alpha_1,\alpha_2 )$ at $(u_1,u_2)$ on each singular set by 
\begin{align*}
& \Phi _2^{\rm Rev} (u_1,u_2;\alpha_1,\alpha_2  )=\lim_{s_2\to u_2}\lim_{s_1\to u_1}\Phi _2 (s_1,s_2;\alpha_1,\alpha_2  ).
\end{align*}
Let $(s)_k:=s(s+1)(s+2)\cdots (s+k-1)$. Then we obtained

\begin{prop}[\cite{MNT2021}\ Proposition 5.1] \label{spv1} Let $m,n\in \mathbb{N}_0$ with $2\mid (m+n)$ and assume $m\geq 1$ when $n\geq 2$. Then 
\begin{align*}
\Phi _2^{\rm Rev} (-m,-n;1,\Lambda )
&\ =\frac {B_{m+n+2}}{(n+1)(m+n+2)}+(\log 2\pi ) \frac{B_{m+n+1}}{m+n+1}\\
&\ \ -\sum _{k=1 \atop k:{\rm odd}}^n \binom n k \frac {k+1}{B_{k+1}}\zeta '(-k)\frac {B_{m+n-k+1}}{m+n-k+1}\\
&\ \ +\sum _{k=1\atop k:{\rm even}}^n  \binom n k \frac {B_{m+n-k+1}}{m+n-k+1}\left( -a_k+k!b_k\right) \\
& \qquad\quad -\frac{(-1)^m m!n!}{(m+n+1)!} M(-m-n-1).
\end{align*}
\end{prop}
\begin{remark}\label{exceptionalcase}
In the above proposition, the case $m=0$, $n\geq 2$ is excluded.   In this case
we can see that $\Phi _2^{\rm Rev} (-m,-n;1,\Lambda )$ is not convergent.
See Section \ref{sec-6}.
\end{remark}

Next we consider the case $\Phi(s;\alpha)=1$, namely $\widetilde{\alpha}=\mu$.
In this case, $c_k=c_k(\alpha)$ is defined by 
\begin{align}
&c_k=\lim_{s\to -k}\ \frac{d}{ds}\frac{s+k}{\zeta(s)}=-\frac{\zeta''(-k)}{2(\zeta'(-k))^2}\quad (k:\textrm{even}). \label{def-ck-2}
\end{align}
From \eqref{mainthm2}, 
it follows that
\begin{align}
&\Phi _2(s_1,s_2;1,\mu )\label{mainthm2-2} \\
&=-2\zeta (s_1+s_2)-\sum _{k=1\atop k:{\rm odd}}^{N-1}\binom {-s_2}k \frac {(k+1)\zeta (s_1+s_2+k)}{B_{k+1}}\nonumber \\
&+\sum _{k=1\atop k:{\rm even}}^{N-1}\bigg[\binom {-s_2}k\bigg\{ \frac {(-1)^{k/2}2(2\pi )^{k}}{\zeta(1+k)}\left(b_k \zeta(s_1+s_2+k)-\frac{\zeta'(s_1+s_2+k)}{k!}\right)\notag\\
& \quad  +c_k\zeta(s_1+s_2+k)\notag\bigg\} -\frac{\Gamma'(s_2+k)}{\Gamma(s_2)}\frac {(-1)^{k/2}2(2\pi )^{k}}{(k!)^2 \zeta (1+k)}\zeta (s_1+s_2+k)\bigg]\nonumber \\
&+\frac 1 {\Gamma (s_2)}\sum _{n=1}^\infty \Gamma (s_2-\rho _n)\Gamma (\rho _n)\frac {\zeta (s_1+s_2-\rho _n)}{\zeta '(\rho _n)}\nonumber \\
&+\frac 1 {2\pi i\Gamma (s_2)}\int _{(N-\eta)}\Gamma (s_2+z)\Gamma (-z)\frac {\zeta (s_1+s_2+z)}{\zeta (-z)}dz \notag
\end{align}
for $N\in \mathbb N$, where $\eta$ is a small positive number, $b_k$ and $c_k$ are defined by \eqref{def-bk} and \eqref{def-ck-2} (see \cite[(6.2)]{MNT2021}). 

Since $\Phi(s;\alpha)=1$ has no pole, \eqref{coro-2} in this case implies
\begin{equation*}
\begin{split}
& s_1+s_2=1-k\quad (k\in \mathbb{N}_0),\\
& s_2=-k\quad (k\in \mathbb{N},\ k\geq 2),\\
& s_2=-l+\rho_n\quad (l\in \mathbb{N}_0,\ n\in \mathbb{N}),\\
& s_1+ s_2=1+\rho_n\quad (n\in \mathbb{N}).
\end{split}
\end{equation*}
We can calculate, similar to Proposition \ref{spv1}, the reverse value of 
$\Phi _2(s_1,s_2;1,\mu )$ at $(s_1,s_2)=(-m,-n)$ for $m,n\in \mathbb{N}_0$, 
$2\mid (m+n)$ and the additional condition $m\geq 1$ when $n\geq 2$.  
Also, similar to Remark \ref{zzz}, $\Phi_2(s_1,s_2;1,\mu )$ is not
convergent if $m=0$, $n\geq 2$.

%%%%%%%%%%%%%%%%%%%%%%%%%%%%%%%%%%%%%%%%%%%%%%%%%
\section{Analytic behavior of $\Phi _2 (s_1,s_2;1,\Lambda )$ at nonpositive integer points}\label{sec-3}
%%%%%%%%%%%%%%%%%%%%%%%%%%%%%%%%%%%%%%%%%%%%%%%%%%

In this section, we observe the analytic behavior of 
$\Phi _2 (s_1,s_2;1,\Lambda )$ %and $\Phi _2 (s_1,s_2;1,\widetilde{\alpha} )$
 at $(s_1,s_2)=(-m,-n)$ for $m\in \mathbb{Z}$ and $n \in \mathbb{N}_{0}$ with 
$m \not\equiv n \mod 2$. In fact, we aim to confirm whether the reverse values $\Phi _2^{\rm Rev} (-m,-n;1,\Lambda )$ %and $\Phi _2^{\rm Rev} (-m,-n;1,\widetilde{\alpha} )$
at these points are convergent or not. For this aim, it is necessary to observe their analytic behavior at these points.

We prove the following.

\begin{thm}\label{Theorem-3-1}\ Let $m\in \mathbb{Z}$, $n\in \mathbb{N}_{0}$ and $\ell \in \mathbb{N}$ with $m+n=2\ell-1$. \\
If $m\geq 0$, then
\begin{align}\label{3-1}
&\Phi_2(-m,-n+\varepsilon;1,\Lambda)
  = \frac{1}{\varepsilon}\bigg\{ 
(-1)^n\frac{m!n!}{(m+n+1)!} 
\\
& \qquad +\sum_{k=1 \atop k:\,\text{\rm even}}^{n}\binom{n}{k}\zeta(-m-n+k)\bigg\}+O(1)\quad (\varepsilon \to 0). \notag
\end{align}
If $m \leq -1$, then 
\begin{align}
\Phi_2&(-m,-n+\varepsilon;1,\Lambda) \label{Kocchiga-3-2}\\
& = \frac{2}{\varepsilon^2}\binom{n}{2\ell}+\frac{1}{\varepsilon}
\bigg[ \binom{n}{2\ell}\left\{(-a_{2\ell}+(2\ell)!b_{2\ell})-\sum_{j=0}^{2\ell-1}\frac{1}{n-j}\right\}\notag\\
& \quad +\sum_{k=1 \atop k:\,\text{\rm even}}^{n}\binom{n}{k}\zeta(-m-n+k)+\binom{n}{2\ell}\gamma-\frac{C(2\ell,n)}{(2\ell)!}\bigg] \notag\\
& \quad +O(1)\quad (\varepsilon \to 0), \notag
\end{align}
where $\gamma$ is the Euler constant and $C(2\ell,n)$ is the constant defined in \eqref{3-8} determined by the gamma function (see below).
\end{thm}

\begin{proof}
First we set $(s_1,s_2)=(-m,-n+\varepsilon)$ for a small $\varepsilon>0$ in \eqref{mainthm}. Then 
\begin{align} \label{3-3}
&\Phi _2 (-m,-n+\varepsilon;1,\Lambda ) \\
&=\frac {\zeta (-m-n-1+\varepsilon)}{-n-1+\varepsilon}-(\log 2\pi ) \zeta (-m-n+\varepsilon)\notag\\
&\ +\sum ^{N-1}_{k=1 \atop k:{\rm odd}} \binom {n-\varepsilon}k M(-k)\zeta (-m-n+k+\varepsilon)\nonumber \\
&\ -\sum ^{N-1}_{k=1 \atop k:{\rm even}} \bigg[ \binom {n-\varepsilon}k\left\{ \left( -{a_k}+k!b_k\right)\zeta (-m-n+k+\varepsilon)-\zeta'(-m-n+k+\varepsilon)\right\} \nonumber \\
& \qquad\qquad -\frac {1}{k!}\frac{\Gamma'(-n+k+\varepsilon)}{\Gamma(-n+\varepsilon)}\zeta(-m-n+k+\varepsilon)\bigg]\nonumber\\
&\ -\frac{1}{\Gamma (-n+\varepsilon)}\sum _{h=1}^\infty {\rm ord}(\rho _h)\Gamma (-n-\rho _h+\varepsilon)\Gamma (\rho _h)\zeta (-m-n-\rho_h+\varepsilon)\nonumber  \\
&\ +\frac{1}{2\pi i\Gamma (-n+\varepsilon)}\int_{(N-\eta)}\Gamma (-n+z+\varepsilon)\Gamma (-z)M(-z)\zeta (-m-n+z+\varepsilon)dz\nonumber \\
& = A_1-A_2+A_3-A_4-A_5+A_6, \notag
\end{align}
say. Since $1/\Gamma(-n+\varepsilon) \to 0$ as $\varepsilon \to 0$, we see that $A_5$ and $A_6$ tend to $0$. Also, letting $\varepsilon \to 0$, we have
\begin{align*}
&A_1 \to \frac {\zeta (-m-n-1)}{-n-1}=\frac {\zeta (-2\ell)}{-n-1}=0,\\
&A_2 \to (\log 2\pi)\zeta(1-2\ell),\\
&A_3 \to \sum ^{N-1}_{k=1 \atop k:{\rm odd}} \binom {n}{k} M(-k)\zeta (1-2\ell+k),
\end{align*}
which are convergent.   The only remaining part is
\begin{align*}
A_4&=\sum ^{N-1}_{k=1 \atop k:{\rm even}} \binom {n-\varepsilon}k \left( -{a_k}+k!b_k\right)\zeta (-m-n+k+\varepsilon)\\
& \quad -\sum ^{N-1}_{k=1 \atop k:{\rm even}}\binom{n-\varepsilon}{k} \zeta'(-m-n+k+\varepsilon) \nonumber \\
& \quad -\sum ^{N-1}_{k=1 \atop k:{\rm even}} \frac {1}{k!}\frac{\Gamma'(-n+k+\varepsilon)}{\Gamma(-n+\varepsilon)}\zeta(-m-n+k+\varepsilon)\nonumber\\
& =A_{41}-A_{42}-A_{43},
\end{align*}
say, where we choose a sufficiently large $N$ such as $2\ell\leq N-1$. 

As for $A_{41}$, since $m+n=2\ell-1$, we can see that 
\begin{align}
A_{41}&=\binom {n-\varepsilon}{2\ell} \left( -a_{2\ell}+(2\ell)!b_{2\ell}\right)\zeta (1+\varepsilon)+O(1)  \label{3-4}\\
&=\binom {n-\varepsilon}{2\ell} \left( -a_{2\ell}+(2\ell)!b_{2\ell}\right)\left(\frac{1}{\varepsilon}+\gamma+O(\varepsilon)\right)+O(1)\notag\\
&=\frac{1}{\varepsilon}\binom {n}{2\ell} \left( -a_{2\ell}+(2\ell)!b_{2\ell}\right)+O(1)\quad (\varepsilon \to 0).\notag
\end{align}
%where $\gamma$ is the Euler constant.
We remark that,
if $m\geq 0$ (namely $n\leq 2\ell-1$), then $\binom{n}{2\ell}=0$ and hence
$A_{41}=O(1)$.

As for $A_{42}$, we can similarly see that
\begin{align*}
A_{42}&=\binom {n-\varepsilon}{2\ell}\zeta'(1+\varepsilon)+O(1)\\
  & =\frac{(n-\varepsilon)(n-1-\varepsilon)\cdots(n-2\ell+1-\varepsilon)}{(2\ell)!}\left(-\frac{1}{\varepsilon^2}+O(1)\right)+O(1).\notag
\end{align*}
For the case $n\geq 2\ell$, namely $m\leq -1$, we can see that
\begin{align}  \label{3-5}
A_{42}&=-\frac{1}{\varepsilon^2}\binom {n}{2\ell} +\frac{1}{\varepsilon}\binom {n}{2\ell}\left(\sum_{j=0}^{2\ell-1}\frac{1}{n-j}\right)+O(1)\quad (\varepsilon \to 0). 
\end{align}
For the case $n\leq 2\ell-1$, namely $m \geq 0$, we have
\begin{align*}
\binom{n-\varepsilon}{2\ell}&=\frac{(n-\varepsilon)\cdots(1-\varepsilon)(-\varepsilon)(-1-\varepsilon)\cdots(n-2\ell+1-\varepsilon)}{(2\ell)!}\\
& =(-\varepsilon)\frac{n!(-1)^{2\ell-n-1}(2\ell-n-1)!}{(2\ell)!}(1+O(\varepsilon))\\
& =\varepsilon(-1)^n\frac{n!(2\ell-n-1)!}{(2\ell)!}(1+O(\varepsilon)).
\end{align*}
Therefore we have
\begin{align}
A_{42}&=-\frac{(-1)^n}{\varepsilon}\frac{n!(2\ell-n-1)!}{(2\ell)!}+O(1)\quad (\varepsilon \to 0). \label{3-6}
\end{align}

Finally we consider 
$$A_{43}=\sum ^{N-1}_{k=1 \atop k:{\rm even}} \frac {1}{k!}\frac{\Gamma'(-n+k+\varepsilon)}{\Gamma(-n+\varepsilon)}\zeta(-m-n+k+\varepsilon).$$
When $k\geq n+1$, we have $\Gamma'(-n+k+\varepsilon)=O(1)$ $(\varepsilon \to 0)$. Hence, by $1/ \Gamma(-n+\varepsilon) \to 0$, we have
\begin{align}\label{333}
\frac{\Gamma'(-n+k+\varepsilon)}{\Gamma(-n+\varepsilon)}\zeta(-m-n+k+\varepsilon)=O(1)\quad (\varepsilon \to 0).
\end{align}
Note that this includes the case $k=m+n+1$ (because the pole of the zeta
factor is cancelled by the gamma factor).

When $k\leq n$, we quote a result given in \cite[(2.18)]{MNT2021}:
\begin{equation}\label{3-7}
\frac{\Gamma'(s+k)}{\Gamma(s)}=\frac{(-1)^{k-1}(k+h)!}{h!}\frac{1}{s+k+h}+O(1) \quad (s\to -k-h)
\end{equation}
for even $k\geq 2$ and $h \in \mathbb{N}_{0}$. Letting $s=-n+\varepsilon$ and $h=n-k\geq 0$ in \eqref{3-7}, we can write
\begin{equation}\label{3-8}
\frac{\Gamma'(-n+k+\varepsilon)}{\Gamma(-n+\varepsilon)}=\frac{1}{\varepsilon}\frac{(-1)^{k-1}n!}{(n-k)!}+C(k,n)+O(\varepsilon) \quad (\varepsilon \to 0)
\end{equation}
with a constant $C(k,n)$. 
The zeta factor $\zeta(-m-n+k+\varepsilon)$ remains finite unless 
$k=2\ell$, while for $k=2\ell$ (which occurs when $2\ell\leq n$), we have
\begin{equation}\label{3-9}
\zeta(-m-n+k+\varepsilon)=\zeta(1+\varepsilon)=\frac{1}{\varepsilon}+\gamma+O(\varepsilon)\quad (\varepsilon \to 0).
\end{equation}
. 

Summarizing the above results, we obtain the following.\\
For the case $n\leq 2\ell-1$, namely $m\geq 0$, we obtain from 
\eqref{333} and \eqref{3-8} that
\begin{align}\label{3-10}
A_{43}& = \sum ^{n}_{k=1 \atop k:{\rm even}} \frac {1}{k!}\left(\frac{1}{\varepsilon}\frac{(-1)^{k-1}n!}{(n-k)!}\right)\zeta(1-2\ell+k+\varepsilon)+O(1)\\
& =\frac{1}{\varepsilon} \sum^{n}_{k=1 \atop k:{\rm even}} (-1)^{k-1}\binom{n}{k}\zeta(1-2\ell+k)+O(1)\quad (\varepsilon \to 0). \notag
\end{align}
For the case $n\geq 2\ell$, namely $m\leq -1$, we obtain from \eqref{333}, \eqref{3-8} and \eqref{3-9} that
\begin{align}\label{3-11}
A_{43}& = \sum ^{n}_{k=1 \atop {k:{\rm even} \atop k\neq 2\ell}} \frac {1}{k!}\left(\frac{1}{\varepsilon}\frac{(-1)^{k-1}n!}{(n-k)!}\right)\zeta(1-2\ell+k+\varepsilon)\\
&\ \ +\frac{1}{(2\ell)!}\left(\frac{1}{\varepsilon}\frac{(-1)^{2\ell-1}n!}{(n-2\ell)!}+C(2\ell,n)+O(\varepsilon)\right)\notag\\
& \qquad \times \left(\frac{1}{\varepsilon}+\gamma+O(\varepsilon)\right)+O(1) \notag\\
& = -\frac{1}{\varepsilon^2}\binom{n}{2\ell} \notag\\
&\ \  -\frac{1}{\varepsilon}\bigg(\sum ^{n}_{k=1 \atop {k:{\rm even} \atop k\neq 2\ell}} \binom{n}{k}\zeta(1-2\ell+k)+\binom{n}{2\ell}\gamma-\frac{C(2\ell,n)}{(2\ell)!}\bigg)\notag\\
&\ \ +O(1)\quad (\varepsilon \to 0). \notag
\end{align}
Combining \eqref{3-4}-\eqref{3-6}, \eqref{3-10} and \eqref{3-11}, we obtain the proof of Theorem \ref{Theorem-3-1}. 
\end{proof}

\begin{remark}\label{Remark-3-2}\ 
As mentioned before, the condition $n \geq 2\ell$ implies $m \leq -1$. Let $h=-m\in \mathbb{N}$. %Then the condition that $m+n$ is odd positive implies that $n-h$ is odd positive. 
Then
\eqref{Kocchiga-3-2} in Theorem \ref{Theorem-3-1} shows that for $h,\,n\in \mathbb{N}$ such that $n-h$ is odd positive, $\Phi_2^{\rm rev}(h,-n;1,\Lambda)$ is clearly not convergent.

On the other hand, the condition $n \leq 2\ell-1$ implies that $m \geq 0$. 
Hence 
\eqref{3-1} in Theorem \ref{Theorem-3-1} shows that for $m,\,n\in \mathbb{Z}_{\geq 0}$ such that $m+n$ is odd positive, $\Phi_2^{\rm rev}(-m,-n;1,\Lambda)$ is convergent if and only if the residue 
\begin{align}\label{3-12}
%(-1)^n\frac{n! (2\ell-1-n)!}{(2\ell)!} +\sum_{k=1 \atop k:\,\text{\rm even}}^{n}\binom{n}{k}\zeta(1-2\ell+k)\\
R(-m,-n):=(-1)^n\frac{m! n!}{(m+n+1)!} +\sum_{k=1 \atop k:\,\text{\rm even}}^{n}\binom{n}{k}\zeta(k-m-n)
\end{align}
equals to zero. We will give further results on the values of
$R(-m,-n)$ in the next section.
\end{remark}

%%%%%%%%%%%%%%%%%%%%%%%%%%%%%%%%%%%%%%
\section{Reciprocity relations for residues}\label{sec-4}
%%%%%%%%%%%%%%%%%%%%%%%%%%%%%%%%%%%%%%

In this section, we will evaluate the values of $R(-m,-n)$ in several cases.
In particular, we will prove certain reciprocity relation between
$R(-m,-n)$ and $R(-n,-m)$ 
for $m,n \in \mathbb{N}$ with $2\nmid (m+n)$. 
Since 
\begin{align}\label{zzz}
\zeta(1-r)=(-1)^{r-1}B_{r}/r \qquad(r\in \mathbb{N})
\end{align}
(see \cite[Chapter II]{Titch}), $R(-m,-n)$ defined by \eqref{3-12} can also be written as
\begin{align}\label{4-1}
R(-m,-n)%=(-1)^n\frac{m!n!}{(m+n+1)!} +\sum_{k=1 \atop k:\,\text{\rm even}}^{n}\binom{n}{k}\zeta(-m-n+k)\\
=(-1)^n\frac{m!n!}{(m+n+1)!} -\sum_{j=1}^{[n/2]}\binom{n}{2j}\frac{B_{m+n+1-2j}}{m+n+1-2j}.
\end{align}
We first prove the following.
\begin{thm}\label{Theorem-4-1}\ 
For $N\in \mathbb{N}_{0}$,
\begin{align}\label{4-2}
& R(-1,-2N)=
\begin{cases}
0 & (N\geq 1),\\
\frac{1}{2} & (N=0).
\end{cases}
\end{align}
Hence $\Phi_2^{\rm Rev}(-1,-2N;1,\Lambda)$ is convergent for $N\geq 1$ and $\Phi_2^{\rm Rev}(-1,0;1,\Lambda)$ is not convergent. 
\end{thm}

In order to prove this result, we recall the following reciprocity relation among Bernoulli numbers. This result was first proved by Saalsch\"utz \cite{Saal1892}, and was recovered by Gelfand \cite{Gelfand1967} (see also Agoh-Dilcher \cite{AgohD2008} for the details).

\begin{lem}\label{Lemma-4-2}\ 
For $p,q\in \mathbb{N}_{ 0}$, 
\begin{align}\label{4-3}
& (-1)^{p+1}\sum_{l=0}^{q}\binom{q}{l}\frac{B_{p+1+l}}{p+1+l}+(-1)^{q+1}\sum_{l=0}^{p}\binom{p}{l}\frac{B_{q+1+l}}{q+1+l}\\
&\quad  =\frac{p!q!}{(p+q+1)!}. \notag
\end{align}
\end{lem}

Using this relation, we will prove  Theorem \ref{Theorem-4-1}.

\begin{proof}[Proof of Theorem \ref{Theorem-4-1}]
By \eqref{4-1} with $(m,n)=(1,2\ell-2)$, we have
\begin{align*}%\label{4-1}
R(-1,-2\ell+2)&=\frac{(2\ell-2)!}{(2\ell)!} -\sum_{j=1}^{\ell-1}\binom{2\ell-2}{2j}\frac{B_{2\ell-2j}}{2\ell-2j}. \notag 
\end{align*}
Setting $N=\ell-1\geq 0$ and $h=\ell-j-1$, we have 
\begin{align}\label{rrr}
& R(-1,-2N)=\frac{(2N)!}{(2N+2)!}-\sum_{h=0}^{N-1}\binom{2N}{2h}\frac{B_{2h+2}}{2h+2}.
\end{align}
When $N=0$, \eqref{rrr} implies $R(-1,0)=1/2$. When $N\geq 1$, 
setting $(p,q)=(2N,1)$ in Lemma \ref{Lemma-4-2}, we have 
$$-\frac{B_{2N+2}}{2N+2}+\sum_{h=0}^{N}\binom{2N}{2h}\frac{B_{2h+2}}{2h+2}=\frac{(2N)!}{(2N+2)!},$$
because $B_{2r+1}=0$ $(r \geq 1)$. Combining this with \eqref{rrr}, 
we complete the proof. 
\end{proof}

Next we aim to prove the following reciprocity relation among $R(-m,-n)$. 

\begin{thm}\label{Theorem-4-3}\ 
For $m,n \in \mathbb{N}$ with $2\nmid (m+n)$, 
\begin{align}\label{4-4}
& (-1)^n R(-m,-n)+(-1)^m R(-n,-m)=\frac{m!n!}{(m+n+1)!}.
\end{align}
\end{thm}

\begin{proof}
Using \eqref{zzz} we see that the left-hand side of \eqref{4-3} 
(with $(p,q)=(m,n)\in\mathbb{N}^2$) equals to
\begin{align*}
& (-1)^{m+1}\sum_{l=0}^{n}\binom{n}{l}(-1)^{m+l}\zeta(-m-l)+(-1)^{n+1}\sum_{l=0}^{m}\binom{m}{l}(-1)^{n+l}\zeta(-n-l).
\end{align*}
By setting $j=n-l$ and $j=m-l$ in the first and the second terms, respectively, 
this can be written as 
\begin{align*}
& \sum_{j=0}^{n}\binom{n}{j}(-1)^{n-j+1}\zeta(-m-n+j)+\sum_{j=0}^{m}\binom{m}{j}(-1)^{m-j+1}\zeta(-m-n+j).
\end{align*}
Now assume $2\nmid (m+n)$.  Then the two terms corresponding to $j=0$ are
cancelled with each other, and hence,
\eqref{4-3} implies that
%\begin{align*}%\label{4-5}
%& \sum_{j=0}^{n}\binom{n}{j}(-1)^{n-j+1}\zeta(-m-n+j) \\
%& \quad +\sum_{j=0}^{m}\binom{m}{j}(-1)^{m-j+1}\zeta(-m-n+j)=\frac{m!n!}{(m+n+1)!}. \notag
%\end{align*}
%Since the terms with $j=0$ on the left-hand side are cancelled because of $2\nmid (m+n)$, we obtain
\begin{align}\label{4-5}
& (-1)^{n-1}\sum_{j=1}^{n}\binom{n}{j}(-1)^{j}\zeta(-m-n+j) \\
& \quad +(-1)^{m-1}\sum_{j=1}^{m}\binom{m}{j}(-1)^{j}\zeta(-m-n+j)=\frac{m!n!}{(m+n+1)!}. \notag
\end{align}
From \eqref{3-12}, we have
\begin{align}\label{4-6}
R(-m,-n)&=(-1)^n\frac{m!n!}{(m+n+1)!} +\sum_{k=1 \atop k:\,\text{\rm even}}^{n}\binom{n}{k}\zeta(-m-n+k)\\
&=(-1)^n\frac{m!n!}{(m+n+1)!} +\sum_{k=1}^{n}\binom{n}{k}(-1)^{k}\zeta(-m-n+k), \notag
\end{align}
because $\zeta(-m-n+k)=0$ for any odd $k$ with $k\leq n$ and $m\geq 1$. In fact, $-m-n+k$ is even and negative in this case. From \eqref{4-5} and \eqref{4-6}, we have
\begin{align*}
& (-1)^{n-1}\left\{R(-m,-n)-(-1)^n\frac{m!n!}{(m+n+1)!}\right\}\\
& +(-1)^{m-1}\left\{R(-n,-m)-(-1)^m\frac{m!n!}{(m+n+1)!}\right\}=\frac{m!n!}{(m+n+1)!}.
\end{align*}
Thus we complete the proof.
\end{proof}

Combining Theorems \ref{Theorem-4-1} and \ref{Theorem-4-3}, we obtain the following.

\begin{cor} \label{Cor-4-4}\ For $N \in \mathbb{N}$, 
$$ R(-2N,-1)=-\frac{1}{(2N+1)(2N+2)}.$$
\end{cor}

Finally we prove the following.

\begin{prop} \label{Prop-4-5}\ For $N \in \mathbb{N}$, 
$$R(0,-2N-1)=-\frac{1}{2}.$$
\end{prop}

\begin{proof}
By \eqref{4-1} with $(m,n)=(0,2N-1)$, we have
\begin{align*}
R(0,1-2N)&=-\frac{1}{2N} -\sum_{j=1}^{N-1}\binom{2N-1}{2j}\frac{B_{2N-2j}}{2N-2j}.
\end{align*}
We obtain from \eqref{4-3} with $(p,q)=(0,2N-1)$ that
\begin{align*}
\frac{1}{2N}&=-B_1-\sum_{\mu=0}^{N-1}\binom{2N-1}{2\mu+1}\frac{B_{2\mu+2}}{2\mu+2}+\frac{B_{2N}}{2N}\\
& = -B_1-\sum_{\mu=0}^{N-2}\binom{2N-1}{2\mu+1}\frac{B_{2\mu+2}}{2\mu+2}\\
& =\frac{1}{2}-\sum_{j=1}^{N-1}\binom{2N-1}{2j}\frac{B_{2N-2j}}{2N-2j},
\end{align*}
by setting $j=N-1-\mu$. Combining these results, we complete the proof.
\end{proof}

%%%%%%%%%%%%%%%%%%%%%%%%%%%%%%%%%%%%%%
\section{Residues of $\Phi_2(s_1,s_2;1,\widetilde{\alpha})$} \label{sec-5}
%%%%%%%%%%%%%%%%%%%%%%%%%%%%%%%%%%%%%%

In this section, similar to Section \ref{sec-3}, we observe the analytic behavior of 
$\Phi _2 (s_1,s_2;1,\widetilde{\alpha} )$ 
 at $(s_1,s_2)=(-m,-n)$ for $m\in \mathbb{Z}$ and $n \in \mathbb{N}_{0}$ with 
$m \not\equiv n \mod 2$, under Assumption \ref{Ass-2}.

Let $\alpha\in \mathcal{A}$ satisfying three conditions (I)-(III) in Section \ref{sec-1}. Let $\delta=\delta(\alpha)>0$ be defined in (I). 
Furthermore, for even $k\in \mathbb{N}$, let 
\begin{equation*}
D_k=D_k^{(\alpha)}=\frac{\Phi(-k;\alpha)}{\zeta'(-k)}=\frac{(-1)^{k/2}2(2\pi)^k\Phi(-k;\alpha)}{k! \zeta(k+1)}. 
\end{equation*}
We prove the following.

\begin{thm}\label{Theorem-5-1}\ Let $m\in \mathbb{Z}$, $n\in \mathbb{N}_{0}$ and $\ell \in \mathbb{N}$ with $m+n=2\ell-1$.\\
Under Assumption \ref{Ass-2}, we have
\begin{align}
\Phi_2&(-m,-n+\varepsilon;1,\widetilde{\alpha})  = \frac{1}{\varepsilon}\bigg\{ (-1)^nD_{2\ell}\frac{m!n!}{(m+n+1)!} \label{5-1}\\
& \qquad +\sum_{k=1 \atop k:\,\text{\rm even}}^{n}(-1)^k\binom{n}{k}D_k\zeta(-m-n+k)\bigg\}+O(1)\quad (\varepsilon \to 0) \notag
\end{align}
if $m\geq 0$, and
\begin{align}\label{3-2} 
\Phi_2&(-m,-n+\varepsilon;1,\widetilde{\alpha})  = \frac{2D_{2\ell}}{\varepsilon^2}\binom{n}{2\ell}+O(1/\varepsilon)\quad (\varepsilon \to 0)%\notag
\end{align}
if $m \leq -1$.
\end{thm}

Note that the reverse value \eqref{5-1} is convergent if and only if the residue on the right-hand side vanishes.
Also, if $D_{2\ell}\neq 0$ (namely $\Phi(-2\ell;\alpha)\neq 0$), the reverse value \eqref{3-2} is not convergent.

\begin{proof}[Proof of Theorem \ref{Theorem-5-1}]\ 
Since the proof is quite similar to the case of $\Phi _2 (s_1,s_2;1,\Lambda)$, we will give a sketch of the proof.  

Set $(s_1,s_2)=(-m,-n+\varepsilon)$ for a small $\varepsilon>0$ in 
\eqref{mainthm2}. Then 
\begin{align*} %\label{5-3}
&\Phi _2 (-m,-n+\varepsilon;1,\widetilde{\alpha} ) \\
&=\frac{\Gamma(-n+\varepsilon-\delta)\Gamma(\delta)}{\Gamma(-n+\varepsilon)} \Res_{s=\delta}\ \left(\frac{\Phi(s;\alpha)}{\zeta(s)}\right)\zeta(-m-n+\varepsilon-\delta) \notag\\
& \ -2\Phi(0;\alpha)\zeta (-m-n+\varepsilon)-\sum _{k=1 \atop k:{\rm odd}}^{N-1}\binom {n-\varepsilon}k \frac {(k+1)\Phi(-k;\alpha)}{B_{k+1}}\zeta (-m-n+\varepsilon+k)\nonumber \\
&\ +\sum _{k=1\atop k:{\rm even}}^{N-1}  \bigg[ \binom {n-\varepsilon}k\bigg\{ \frac {(-1)^{k/2}2(2\pi )^{k}\Phi(-k;\alpha)}{\zeta (1+k)}\notag \\
& \ \ \times \left(b_k \zeta(-m-n+\varepsilon+k)-\frac{1}{k!}\zeta'(-m-n+\varepsilon+k)\right)+c_k(\alpha)\zeta(-m-n+\varepsilon+k)\bigg\}\nonumber \\
& \qquad -\frac{\Gamma'(-n+\varepsilon+k)}{\Gamma(-n+\varepsilon)}\frac {(-1)^{k/2}2(2\pi )^{k}\Phi(-k;\alpha)}{(k!)^2\zeta (1+k)}\zeta(-m-n+\varepsilon+k)\bigg]\notag\\
&\ +\frac 1 {\Gamma (-n+\varepsilon)}\sum _{n=1}^\infty \Gamma (-n+\varepsilon-\rho _n)\Gamma (\rho _n)\frac {\Phi(\rho_n;\alpha)}{\zeta'(\rho _n)}\zeta (-m-n+\varepsilon-\rho _n)\nonumber \\
&\ +\frac 1 {2\pi i\Gamma (-n+\varepsilon)}\int _{(N-\eta )}\Gamma (-n+\varepsilon+z)\Gamma (-z)\Phi(-z;\widetilde{\alpha})\zeta (-m-n+\varepsilon+z)dz\nonumber\\
& = A_1-A_2-A_3+A_4+A_5+A_6, \notag
\end{align*}
say. By just the same way as in the case of $\Phi _2 (-m,-n+\varepsilon;1,\Lambda)$, we can check that $A_1,\,A_2,\,A_3,\,A_5$ and $A_6$ are $O(1)$ as $\varepsilon \to 0$. By the definition of $D_k$, we have
\begin{align*}
A_4&=\sum _{k=1\atop k:{\rm even}}^{N-1}  \binom {n-\varepsilon}{k} k!D_{k}b_k \zeta(-m-n+\varepsilon+k)\\
& -\sum _{k=1\atop k:{\rm even}}^{N-1} \binom {n-\varepsilon}{k} D_{k}\zeta'(-m-n+\varepsilon+k)\\
& +\sum _{k=1\atop k:{\rm even}}^{N-1}  \binom {n-\varepsilon}{k}c_k(\alpha)\zeta(-m-n+\varepsilon+k)\nonumber \\
& -\sum _{k=1\atop k:{\rm even}}^{N-1}\frac{\Gamma'(-n+\varepsilon+k)}{\Gamma(-n+\varepsilon)}\frac {D_k}{k!}\zeta(-m-n+\varepsilon+k)\notag\\
& =A_{41}-A_{42}+A_{43}-A_{44},
\end{align*}
say, where we choose a sufficiently large $N$. Let $m+n=2\ell-1$ $(\ell \in \mathbb{N})$. Similarly as in the proof of Theorem \ref{Theorem-3-1},
we obtain
\begin{equation*}
A_{41}=
\begin{cases}
\frac{1}{\varepsilon}\binom{n}{2\ell}(2\ell)!D_{2\ell}b_{2\ell}+O(1) & (m\leq -1),\\
O(1) & (m\geq 0),
\end{cases}
\end{equation*}
\begin{equation*}
A_{42}=
\begin{cases}
-\frac{1}{\varepsilon^2}\binom{n}{2\ell}D_{2\ell}+O(1/\varepsilon) & (m\leq -1),\\
\frac{1}{\varepsilon}(-1)^{n-1}D_{2\ell}\frac{m!n!}{(m+n+1)!}+O(1) & (m\geq 0),
\end{cases}
\end{equation*}
\begin{equation*}
A_{43}=
\begin{cases}
\frac{1}{\varepsilon}\binom{n}{2\ell}c_{2\ell}(\alpha)+O(1) & (m\leq -1),\\
O(1) & (m\geq 0),
\end{cases}
\end{equation*}
\begin{equation*}
A_{44}=
\begin{cases}
-\frac{1}{\varepsilon^2}\binom{n}{2\ell}D_{2\ell}+O(1/\varepsilon) & (m\leq -1),\\
\frac{1}{\varepsilon}\sum_{1\leq k \leq n \atop k:{\rm even}}(-1)^{k-1}\binom{n}{k}D_k\zeta(-m-n+k)+O(1) & (m\geq 0).
\end{cases}
\end{equation*}
Thus we complete the proof.
\end{proof}

%%%%%%%%%%%%%%%%%%%%%%%%%%%%%%%%%
\section{Corrigendum and addendum to \cite{MNT2021}} \label{sec-6}
%%%%%%%%%%%%%%%%%%%%%%%%%%%%%%%%

\noindent
Corrigendum to \cite{MNT2021}:

\begin{enumerate}[$\bullet$]
\item In \cite[Remark 5.2]{MNT2021}, we mentioned that if $m+n$ is odd with $n\geq 2$, then $\Phi_2(s_1,s_2;1,\Lambda)$ is not convergent as ${s_1\to -m}$ and ${s_2\to -n}$. However, as stated in Sections \ref{sec-3} and \ref{sec-4}
(especially Theorem \ref{Theorem-4-1}), it sometimes happens that $\Phi_2^{\rm Rev}(-m,-n;1,\Lambda)$ is convergent. 
\item In \cite[Example 5.3]{MNT2021}, we listed 
$$\Phi _2^{\rm Rev} (-1,0;1,\Lambda )=\frac 1 {12}\log 2\pi -\frac{3}{4}-\frac{\zeta''(-2)}{4\zeta'(-2)}+\frac{\gamma}{2},$$
which is not correct. It follows from Theorem \ref{Theorem-4-1} that 
$\Phi _2^{\rm Rev} (-1,0;1,\Lambda )$ is not convergent.
\item In \cite[Example 6.1]{MNT2021}, we listed 
$$\Phi _2^{\rm Rev}(-1,0;1,\mu )=\frac 1 6 -\frac {2\pi ^2}{\zeta (3)}\left(\frac{3}{2}-\gamma\right)-\frac{\zeta''(-2)}{4\zeta'(-2)^2},$$
which is not correct. In fact, as noted in \eqref{mu-Dir}, we see that $\widetilde{\alpha}=\mu$ implies $\Phi(s;\alpha)=1$, hence 
\begin{equation*}
D_k=D_k^{(\alpha)}=\frac{1}{\zeta'(-k)}=\frac{(-1)^{k/2}2(2\pi)^k}{k! \zeta(k+1)}\neq 0
\end{equation*}
for even $k\geq 0$. Therefore it follows from Theorem \ref{Theorem-5-1} with $\widetilde{\alpha}=\mu$ that 
\begin{align}
\Phi_2&(-1,\varepsilon;1,\mu)  = \frac{1}{\varepsilon}\frac{D_{2}}{2} +O(1)\qquad (\varepsilon \to 0),
\end{align}
which implies that $\Phi _2^{\rm Rev}(-1,0;1,\mu )$ is not convergent.
\end{enumerate}

\ 

\noindent
Addendum to \cite{MNT2021}:

In \cite[Proposition 5.1]{MNT2021} (= Proposition \ref{spv1}), 
we considered the situation $m,n\in\mathbb{N}_0$ with $2|(m+n)$, but
the case
$m=0$, $n\geq 2$ was excluded.    As mentioned in Remark 
\ref{exceptionalcase}, in this excluded case, 
$\Phi _2^{\rm Rev} (-m,-n;1,\Lambda )$
is indeed not convergent.

In fact, in the proof of \cite[Proposition 5.1]{MNT2021}, we need to
compute
$$
-\frac{1}{k!}\frac{\Gamma'(-n+k+\varepsilon)}{\Gamma(-n+\varepsilon)}
\zeta(-m-n+k+\varepsilon)
$$
for even $k\leq n$.   Using \eqref{3-8} we find that the above is
\begin{align}\label{add}
&=(-1)^k\binom{n}{k}\zeta(-m-n+k)\frac{1}{\varepsilon}
-\frac{C(k,n)}{k!}\zeta(-m-n+k)\\
&\quad +(-1)^k\binom{n}{k}\zeta'(-m-n+k)
+O(\varepsilon).\notag
\end{align}
If $m\geq 1$, then $-m-n+k$ is an even negative integer (because
$k$ and $m+n$ are even, $k\leq n$), and so $\zeta(-m-n+k)=0$.
Therefore \eqref{add} tends to
$\binom{n}{k}\zeta'(-m-n+k)$ as $\varepsilon \to 0$, which is the 
formula on the last line of
\cite[p.452]{MNT2021}.
However if $m=0$, then $n$ is even, and 
$\zeta(-m-n+k)=\zeta(0)=-1/2$ for $k=n$.
Therefore in this case \eqref{add} is $-(2\varepsilon)^{-1}+O(1)$, 
which is not convergent.

Similar phenomenon also happens for $\Phi_2(s_1,s_2;1,\mu)$.

\

\section*{Acknowledgments.}
The authors are sincerely grateful to Professor Masatoshi Suzuki who first
pointed out a mistake in \cite{MNT2021}.

\end{document}